\documentclass[graybox]{svmult}

\usepackage{type1cm}        % activate if the above 3 fonts are
\usepackage{makeidx}         % allows index generation
\usepackage{graphicx}        % standard LaTeX graphics tool
\usepackage{multicol}        % used for the two-column index
\usepackage[bottom]{footmisc}% places footnotes at page bottom

\usepackage{newtxtext}       % 
\usepackage{newtxmath}       % selects Times Roman as basic font

\usepackage{caption}

\usepackage{todonotes,cite}

\newcommand{\bA}{\ensuremath{\mathbf{A}}}
\newcommand{\bB}{\ensuremath{\mathbf{B}}}
\newcommand{\bC}{\ensuremath{\mathbf{C}}}

\newcommand{\bF}{\ensuremath{\mathbf{F}}}
\newcommand{\bG}{\ensuremath{\mathbf{G}}}
\newcommand{\bH}{\ensuremath{\mathbf{H}}}
\newcommand{\bI}{\ensuremath{\mathbf{I}}}

\newcommand{\bM}{\ensuremath{\mathbf{M}}}
\newcommand{\bN}{\ensuremath{\mathbf{N}}}

\newcommand{\bP}{\ensuremath{\mathbf{P}}}
\newcommand{\bQ}{\ensuremath{\mathbf{Q}}}

\newcommand{\bV}{\ensuremath{\mathbf{V}}}
\newcommand{\bW}{\ensuremath{\mathbf{W}}}

\newcommand{\bLamb}{\ensuremath{\mathbf{\Lambda}}}

\newcommand{\bb}{\ensuremath{\mathbf{b}}}
\newcommand{\bc}{\ensuremath{\mathbf{c}}}

\newcommand{\bff}{\ensuremath{\mathbf{f}}}
\newcommand{\bg}{\ensuremath{\mathbf{g}}}
\newcommand{\bh}{\ensuremath{\mathbf{h}}}

\newcommand{\bu}{\ensuremath{\mathbf{u}}}
\newcommand{\bv}{\ensuremath{\mathbf{v}}}
\newcommand{\bw}{\ensuremath{\mathbf{w}}}
\newcommand{\bx}{\ensuremath{\mathbf{x}}}
\newcommand{\by}{\ensuremath{\mathbf{y}}}

\newcommand{\Ldir}{{\boldsymbol{\ell}}}
\newcommand{\Rdir}{{\boldsymbol{r}}}

\newcommand{\whA}{\ensuremath{\widehat{\mathbf{A}}}}
\newcommand{\whB}{\ensuremath{\widehat{\mathbf{B}}}}
\newcommand{\whC}{\ensuremath{\widehat{\mathbf{C}}}}

\newcommand{\whF}{\ensuremath{\widehat{\mathbf{F}}}}
\newcommand{\whG}{\ensuremath{\widehat{\mathbf{G}}}}

\newcommand{\whM}{\ensuremath{\widehat{\mathbf{M}}}}
\newcommand{\whN}{\ensuremath{\widehat{\mathbf{N}}}}
\newcommand{\whP}{\ensuremath{\widehat{\mathbf{P}}}}
\newcommand{\whQ}{\ensuremath{\widehat{\mathbf{Q}}}}
\newcommand{\whV}{\ensuremath{\widehat{\mathbf{V}}}}
\newcommand{\whW}{\ensuremath{\widehat{\mathbf{W}}}}

\newcommand{\whb}{\ensuremath{\widehat{\mathbf{b}}}}
\newcommand{\whc}{\ensuremath{\widehat{\mathbf{c}}}}
\newcommand{\whf}{\ensuremath{\widehat{\mathbf{f}}}}

\newcommand{\whv}{\ensuremath{\widehat{\mathbf{v}}}}
\newcommand{\whw}{\ensuremath{\widehat{\mathbf{w}}}}
\newcommand{\whx}{\ensuremath{\widehat{\mathbf{x}}}}
\newcommand{\why}{\ensuremath{\widehat{\mathbf{y}}}}

\newif\ifmatlab\matlabtrue

\newcommand{\cH}{ {\mathcal H} }

\newcommand\IR{ {\mathbb R}}
\newcommand\IC{ {\mathbb C}}

\newcommand{\htwogap}{\cH_2\text{-gap}}
\newcommand{\htwo}{\cH_2}

\usepackage{savesym}                    
\savesymbol{AND}
\usepackage{algorithmic}
\usepackage{algorithm}                       

\usepackage{euscript,multicol,subfigure}
\usepackage{color}
\usepackage{wrapfig,rotating,calc}
\usepackage{tikz,subfigure}
\usepackage{pgfplots}
\usetikzlibrary{calc} 
\usepackage{graphicx}   

\DeclareMathOperator*{\argmin}{arg\,min}
 
 \usepackage{hyperref}
\hypersetup{ colorlinks=true, linktoc=all,  linkcolor=blue}

\usepackage[pageref]{backref}
\renewcommand*{\backrefalt}[4]{%
\ifcase #1 %
(Not cited)%
\or
(Cited on p.~#2)%
\else
(Cited on pp.~#2)%
\fi
}

\makeindex             % used for the subject index
\begin{document}

\title*{$\mathcal{H}_2$-gap model reduction for stabilizable \\ and detectable systems}
\author{T. Breiten and C.~A. Beattie and S. Gugercin}
\institute{T. Breiten \at Institute of Mathematics and Scientific Computing, University of Graz, 8010 Graz, Austria, \\ \email{tobias.breiten@uni-graz.at} 
\and
C.~A. Beattie \at Department of Mathematics, Virginia Tech, Blacksburg, VA 24061-0123, USA, \\ \email{beattie@vt.edu}
\and S. Gugercin \at Department of Mathematics, Virginia Tech, Blacksburg, VA 24061-0123, USA,\\  \email{gugercin@vt.edu}}
\maketitle

%\abstract*{Model reduction for stabilizable and detectable linear time-invariant control systems is studied. Based on the solution of two algebraic Riccati equations, a modified $\mathcal{H}_2$-error measure is introduced. A pole-residue formula involving the closed-loop eigenvalues is derived and analyzed with respect to interpolation-based model reduction. A numerical algorithm for constructing reduced-order models is proposed and shown to be computable without solving the Riccati equations of the original large-scale system. Numerical examples for an unstable convection diffusion equation and a linearized flow problem are given. The new approach is compared to existing model reduction techniques for unstable systems.}
%
%\abstract{Model reduction for stabilizable and detectable linear time-invariant control systems is studied. Based on the solution of two algebraic Riccati equations, a modified $\mathcal{H}_2$-error measure is introduced. A pole-residue formula involving the closed-loop eigenvalues is derived and analyzed with respect to interpolation-based model reduction. A numerical algorithm for constructing reduced-order models is proposed and shown to be computable without solving the Riccati equations of the original large-scale system. Numerical examples for an unstable convection diffusion equation and a linearized flow problem are given. The new approach is compared to existing model reduction techniques for unstable systems.}

\abstract*{We formulate here an approach to model reduction that is well-suited for linear time-invariant control systems that are stabilizable and detectable but may otherwise be unstable.   We introduce a modified $\mathcal{H}_2$-error metric, the 
$\mathcal{H}_2$-gap, that provides an effective measure of model fidelity in this setting.   While the direct evaluation of the $\mathcal{H}_2$-gap requires the solutions of a pair of algebraic Riccati equations associated with related closed-loop systems, we are able to work entirely within an interpolatory framework, developing algorithms and supporting analysis that do not reference full-order closed-loop Gramians.    This leads to a computationally effective strategy yielding reduced models designed so that the corresponding reduced closed-loop systems will interpolate the full-order closed-loop system at specially adapted interpolation points, without requiring evaluation of the full-order closed-loop system nor even computation of the feedback law that determines it.   
The analytical framework and computational algorithm presented here provides an effective new approach toward constructing reduced-order models for unstable systems.   Numerical examples for an unstable convection diffusion equation and a linearized incompressible Navier-Stokes equation illustrate the effectiveness of this approach.}

\abstract{We formulate here an approach to model reduction that is well-suited for linear time-invariant control systems that are stabilizable and detectable but may otherwise be unstable.   We introduce a modified $\mathcal{H}_2$-error metric, the 
$\mathcal{H}_2$-gap, that provides an effective measure of model fidelity in this setting.   While the direct evaluation of the $\mathcal{H}_2$-gap requires the solutions of a pair of algebraic Riccati equations associated with related closed-loop systems, we are able to work entirely within an interpolatory framework, developing algorithms and supporting analysis that do not reference full-order closed-loop Gramians.    This leads to a computationally effective strategy yielding reduced models designed so that the corresponding reduced closed-loop systems will interpolate the full-order closed-loop system at specially adapted interpolation points, without requiring evaluation of the full-order closed-loop system nor even computation of the feedback law that determines it.   
The analytical framework and computational algorithm presented here provides an effective new approach toward constructing reduced-order models for unstable systems.   Numerical examples for an unstable convection diffusion equation and a linearized incompressible Navier-Stokes equation illustrate the effectiveness of this approach.}

\section{Introduction} 
\label{sec:intro} 

Consider a linear time-invariant (LTI) control system
\begin{equation}\label{eq:sys_o}
  \begin{aligned}
  \dot{\bx}(t)&= \bA \bx(t) + \bB \bu(t), \quad \bx(0)=\bx_0, \\
  \by(t)     &= \bC \bx(t),
  \end{aligned}
\end{equation}
where $\bA \in \IR^{n\times n}$ and $\bB \in \IR^{n\times m }$ and $\bC \in 
\IR^{p\times n}.$   The quantities $\bx(t),\bu(t)$ and $\by(t)$ denote respectively the 
state, control and output of the system, viewed as vector-valued functions of time.     
In practical applications, 
for example when \eqref{eq:sys_o} is obtained by a (method-of-lines) semidiscretization of a 
partial differential equation, the dimension $n$ of the state space can remain very large; 
indeed, often large enough to impede subsequent analysis.
In this case, \textit{model reduction} provides useful tools for constructing simpler 
surrogates or \emph{reduced-order models} (ROMs) that produce dynamics analogous to  \eqref{eq:sys_o}:
\begin{equation}\label{eq:sys_r}
  \begin{aligned}
     \dot{\whx}(t)&= \whA \whx(t) + \whB \bu(t), \quad \whx(0)=\whx_0, \\
  \why(t)     &= \whC \whx(t),
  \end{aligned}
\end{equation}
where $\whA \in \IR^{r\times r}$ and $\whB \in \IR^{r\times m }$ and $\whC \in 
\IR^{p\times r}$ are to be determined so that $r\ll n$ and $\why(t)\approx \by(t)$ for all $t\ge 
0$ and all $\bu \in \mathcal{U},$ where, e.g.,  
$\mathcal{U}=L^2(0,\infty;\mathbb R^m).$ 
Assuming null initial conditions for \eqref{eq:sys_o} and \eqref{eq:sys_r}, we may transform \eqref{eq:sys_o} and \eqref{eq:sys_r}
into the frequency domain and introduce the corresponding transfer functions, $\bG(s)=\bC(s\bI-\bA)^{-1}\bB$ and 
$\whG(s)=\whC(s\bI-\whA)^{-1}\whB$, defined for $s\in \mathbb{C}$ with $\mathsf{Re}(s)>\alpha$ for $\alpha\in\mathbb{R}$ sufficiently large.  Observing that $\bG$ and $\widetilde{\bG}$ are each rational functions of $s$, 
the model reduction problem may be restated as a rational approximation problem 
\begin{align}\label{eq:rat_int_prob}
    \widetilde{\bG}  = 
\argmin_{\substack{\mathrm{dim}(\whG)=r}} 
\| \bG-\whG\|
\end{align}
with respect to an appropriately chosen norm. For example, in case that $\bA$ is asymptotically stable 
(i.e., eigenvalues of $\bA$ lie in the open left half-plane, $\mathbb C_-$), the 
most prominent choices are the $\mathcal{H}_\infty$ and the 
$\mathcal{H}_2$-norms. The $\mathcal{H}_\infty$ and the 
$\mathcal{H}_2$ spaces with the associated norms are given by
%, see (\texttt{add references}).   
{\small 
\begin{align*}
%\label{eq:Hardy_spaces}
 \mathcal{H}_2  &= \left\{ \bF\colon \mathbb C^+ \to \mathbb 
C^{p\times m}\ | \ \bF \text{ is analytic  }  ,   \| \bF\| 
_{\mathcal{H}_2}:= \left( \displaystyle\sup_{\sigma >0 } 
\int_{-\infty}^{\infty} \| \bF(\sigma + \imath \omega ) \| _F ^2 \ \mathrm{d} 
\omega \right)^{\frac{1}{2}}< \infty \right \},  \\
 \mathcal{H}_{\infty} &= \left\{ \bF\colon \mathbb C^+ \to \mathbb 
C^{p\times m}\ | \ \bF \text{ is analytic } ,\| \bF\| 
_{\mathcal{H}_\infty}:=\displaystyle \sup_{z \in \mathbb C^+} \| \bF (z) \| _2 < \infty \right\}.
\end{align*}
}
%\begin{align*}
  %\| \bG \| _{\mathcal{H}_2} &:= \left( \displaystyle\sup_{\sigma >0 } 
%\int_{-\infty}^{\infty} \| \bG(\sigma + \imath \omega ) \| _F ^2 \ \mathrm{d} 
%\omega \right)^{\frac{1}{2}}, \quad  
%%%%
  %\| \bG \| _{\mathcal{H}_\infty} := \displaystyle \sup_{z \in \mathbb C^+} \| \bG (z) \| _2.
%\end{align*}

We focus primarily on cases where $\bA$ is \emph{unstable}, i.e., having at least one eigenvalue in $\mathbb C_+$.  
The complementary case where $\bA$ is asymptotically stable is a standard setting that assures whenever 
$\bu \in L^2(0,\infty;\mathbb R^m)$, then also $\bx \in L^2(0,\infty;\mathbb R^n)$ and $\by \in L^2(0,\infty;\mathbb R^p)$. 
When $\bA$ is unstable, the potential for explosive growth of system state and output may reflect features 
that are fundamental to the modeled dynamics,  either naturally so (e.g., linearizations of energetic gyres in ocean circulation) 
or by design (e.g., agile flight vehicles that depend on active control strategies to survive high-speed maneuvers).

For such systems a variety of model reduction strategies have been proposed. One 
possibility is to consider the model reduction problem on a finite time horizon 
$[0,T_{\max}]$; see for example, \cite{GawJ90,GugA03,Kue18,SinG18}. 
In cases that $\sigma(\bA)\cap \imath\IR = \emptyset,$ an 
alternative approach is to decouple the (unstable) system as $\bG=\bG_s + \bG_u$, into a purely stable and 
a purely  anti-stable part, then perform model reduction on each of these two subsystems, 
as was done e.g., using $\mathcal{H}_2$-optimal interpolation techniques in \cite{MagBG10}.  Yet another approach extends balanced truncation to unstable  systems  \cite{zhou1999balanced,barrachina2005parallel,benner2016} using frequency domain definitions of the system Gramians.
% in connection with the  Bernoulli  equations. 

In this work, we also focus on potentially unstable systems, but we do assume 
that $(\bA,\bB)$ is stabilizable and $(\bA,\bC)$ is detectable, a circumstance that commonly occurs. 
It is well known (\cite{NetJB84}), that in this case it is possible to factorize the transfer function as
$\bG=\bM^{-1}\bN,$ with transfer functions $\bM,\bN\in\mathcal{H}_\infty.$  
By introducing a similar representation for the reduced transfer function, 
$\whG=\whM^{-1} \whN,$ we consider an $\mathcal{H}_2$-type best-approximation 
problem having the form 
\begin{align}\label{eq:H2_gap_intro}
    [\widetilde{\bM},\widetilde{\bN}]   = 
\argmin_{\substack{\mathrm{dim}(\whM)=r\\ \mathrm{dim}(\whN)=r}} \left\| 
[\bM,\bN ]-[\whM,\whN] \right\|_{\mathcal{H}_2},
\end{align}
which can be interpretted as seeking reduced order \emph{factors} 
that are as close as possible to the corresponding factors of the original system. 
We refer to the error measure in (\ref{eq:H2_gap_intro}) as the \textit{$\mathcal{H}_2$-gap}.
Our approach is motivated by the method of linear quadratic 
Gaussian (LQG) balancing (\cite{Cur03,JonS83,Mey90}) which introduces a 
similar approximation problem, but using the $\mathcal{H}_\infty$-norm instead of the 
$\mathcal{H}_2$-norm as we have posed it here. 
Besides the evident applicability to unstable systems that is our focus, the $\mathcal{H}_\infty$-induced metric also carries particular significance for associated closed-loop behavior, see, e.g., \cite{Vid84}.

Starting in Section 2, we will introduce our approximation problem in more detail and review some known results on 
left-coprime factorizations for stabilizable and detectable LTI systems. We will show that our approximation problem 
has a natural connection to an $\mathcal{L}_2(\imath \IR)$-model reduction problem.  Section 3 provides a 
pole-residue expansion for the \textit{$\mathcal{H}_2$-gap} 
 and analyzes individual error terms. We suggest a modification of 
the iterative rational Krylov algorithm (IRKA) in order to eliminate portions of 
this error expression. In Section 4, we consider some numerical examples that suggest 
the competitiveness of our approach relative to existing methods.  Conclusions with 
perspectives for future research are given in Section 5.

\section{Left-coprime factorizations and the gap metric}

In this section, we provide additional details for the specific error measure described in \eqref{eq:H2_gap_intro}. Since we do not assume the system matrix $\bA$ to be asymptotically stable, the $\mathcal{H}_2$-norm of $\bG(\cdot)=\bC(\cdot\bI-\bA)^{-1}\bB$ may not be finite. However, as has been shown in, e.g., \cite[Lemma 6.1]{Vid84}, for a stabilizable and detectable 
LTI system, we can select any matrix $\bF$ so that $\bA_\bF:=\bA-\bF \bC$ is 
asymptotically stable, and then 
%\todocb[inline]{In the following, we only use that $(\mathbf{A},\mathbf{C})$ is detectable ($+$ Sherman-Morrison-Woodbury).  Detectability leads to a left factorization while stabilizability leads to a right factorization - I think this might be worth spelling out - but should we separate the assumptions ?  I think for many of the results either one of the assumptions is sufficient. }
\begin{equation}\label{eq:left_coprime}
 \begin{aligned}
  \bM(s) &= \bI - \bC (s\bI -\bA_\bF)^{-1}\bF  \quad
  \bN(s) = \bC (s \bI-\bA_\bF)^{-1}\bB
 \end{aligned}
\end{equation}
defines a left-coprime factorization of $\bG(s)=\bC(s\bI-\bA)^{-1}\bB.$ 
In particular,
\begin{align}  \label{eq:Gcoprime}
  \bG(s) = [\bM(s)]^{-1}\bN(s), \quad \mbox{where}\quad
 [\bM(s)]^{-1} = \bI + \bC (s\bI -\bA)^{-1} \bF .
\end{align}
In the context of LQG-balanced truncation, e.g., in \cite{Mey90}, the stabilizing 
matrix $\bF$ is been chosen as $\bF=\bP \bC^T$ where $\bP=\bP^T\succeq 0$ 
denotes the solution of the algebraic Riccati equation
\begin{equation}\label{eq:Riccati_o}
  \begin{aligned}
%   \bA^T \bQ + \bQ \bA - \bQ \bB \bB^T\bQ + \bC^T\bC &=0, \\
  \bA \bP + \bP \bA^T - \bP \bC^T \bC\bP + \bB \bB^T &=0.
  \end{aligned}
\end{equation}
Note that asymptotic stability of the system matrix $\bA_\bF$ is ensured by the stabilizability and detectability properties of $(\bA,\bB)$ and $(\bA,\bC)$, respectively. Throughout the rest of this work, we assume that $\bF=\bP\bC^T$ and refer to $\bA_\bF=\bA-\bP \bC^T\bC$ as the \textit{closed-loop} system matrix. 
% Due to the stabilizability and detectability assumptions on the system, it is 
% well known that 
% % $\bA_Q = \bA -\bB \bB^T \bQ$ and 
% $\bA_P=\bA-\bP \bC 
% \bC^T$ is asymptotically stable. Following the idea for LQG-balanced 
% truncation (see e.g. \cite{Cur03}), consider the canonical left-coprime 
% factorization of the transfer function, i.e., $\bG(s)=\bM(s)^{-1}\bN(s)$ where
% \begin{equation}\label{eq:left_coprime}
%   \bM(s) = \bI - \bC(s\bI -\bA_P)^{-1} \bP \bC, \quad \bN(s) = \bC(s\bI 
% - \bA_P)^{-1} \bB. 
% \end{equation}

Similarly, for a stabilizable and detectable reduced system 
$(\whA,\whB,\whC),$ we denote by $\whF,\whM,\whN,\whP$  the equivalent 
reduced-order expressions; in particular,
\begin{align} \label{eq:Ghatcoprime}
  \whG(s) = [\whM(s)]^{-1}\whN(s),~~
  \whM(s) = \bI - \whC (s\bI -\whA_{\whF})^{-1} \whF,~~ \mbox{and} ~~
  \whN(s) = \whC (s \bI-\whA_\bF)^{-1}\whB.
\end{align}
 Moreover, it holds that $\bM,\whM \in \mathcal{H}_\infty$ and $\bN,\whN \in 
\mathcal{H}_\infty \cap \mathcal{H}_2.$ 
Define the two new systems 
\begin{align} \label{eqn:GfGfhat}
 \bG_\bF(s):= [\bM(s),\bN(s)], \ \ \whG_\bF(s):= [\whM(s),\whN(s)].
\end{align}
Using \eqref{eq:left_coprime} and \eqref{eqn:GfGfhat}, we can compute  a state-space realization for $\bG_\bF$:
\begin{align}
\bG_\bF(s) & =  [\bM(s), \bN(s)]  =  [  \bI - \bC (s\bI -\bA_\bF)^{-1}\bF,~  \bC (s \bI-\bA_\bF)^{-1}\bB] \nonumber \\
  & =  [\bI, 0] +  \bC (s\bI -\bA_\bF)^{-1} [-\bF,~\bB]. \label{eqn:statespaceG}
\end{align}
A state-space realization for $\whG_\bF$ can be obtained similarly:
\begin{align}
\whG_\bF(s) =  [\bI, 0] +  \whC (s\bI -\whA_\bF)^{-1} [-\whF,~\whB]. \label{eqn:statespaceGhat}
\end{align}
Note that $\bG_\bF$ and $\whG_\bF$ are dynamical systems with $m+p$ inputs and $p$ outputs. LQG balanced truncation  exploits that the Gramians of $\bG_\bF$ and $\whG_\bF$ are closely related to $\bP,\whP$ and their dual counterparts $\bQ,\whQ$, which satisfy
\begin{align*}
 \bA^T \bQ + \bQ \bA - \bQ \bB \bB^T\bQ + \bC^T\bC &=0\quad\mbox{and} \\
 \whA^T \whQ + \whQ \whA - \whQ \whB \whB^T\whQ + \whC^T\whC &=0. 
\end{align*}
In fact, as shown in \cite{Cur03}, the controllability and observability Gramians $\mathcal{L}_c$ and $\mathcal{L}_o$ of $\bG_\bF$ are given by
\begin{align*}
 \mathcal{L}_c  =\bP, \quad \mathcal{L}_o = \bQ(\bI + \bP \bQ)^{-1}.
\end{align*}
Based on these relations, balancing and truncation with respect to $\bP$ and $\bQ$ allows to construct a reduced-order model that satisfies an error bound of the following type 
\begin{equation}\label{eq:gap_metric_error}
  \left\| [\bM,\bN] - [\whM,\whN] \right\| _{\mathcal{H}_\infty}\le 2 
\sum_{i=r+1}^n \sigma_i,
\end{equation}
where $\sigma_i$ are the Hankel singular values of $\bG_\bF.$ 
%\texttt{Maybe more details on the graph operator and the fact that the above norm is the graph norm? There are more details in \cite{Mey90,Vid84,VidSF82}}. \\[1ex]

While balanced truncation is primarily related to the $\mathcal{H}_\infty$-norm, we are interested in the case when the deviation between $\bG_\bF$ and $\whG_\bF$ is measured using the $\mathcal{H}_2$-norm. Thus we define 
\begin{equation}\label{eq:H2_gap}
 \left\| \bG-\whG \right\|_{\mathcal{H}_2\text{-gap}}:= \left \| 
[\bM,\bN]-[\whM,\whN]\right \|_{\mathcal{H}_2}.
\end{equation}
Note that even though $\bM,\whM \notin \mathcal{H}_2,$ the previous expression is well-defined since  
\begin{align*}
 \bM(\cdot)-\whM(\cdot) = \whC (\cdot\bI -\whA_\bF)^{-1}\whF -\bC (\cdot\bI -\bA_\bF)^{-1}\bF \in 
\mathcal{H}_2.
\end{align*}
 If $\bG$ and $\whG$ have no poles on the imaginary axis, we show below that the $\mathcal{H}_2\text{-gap}$, \eqref{eq:H2_gap}, provides a bound to the $\mathcal{L}_2(\imath \mathbb R)$-error.  For this purpose, as in \cite[Section 5]{Ant05} we define 
 {\small
\begin{align*} 
\mathcal{L}_2(\imath \mathbb R)&:= \left\{ \bH\colon \mathbb C \to \mathbb C^{p \times m}\ | \ \bH \text{ is meromorphic },  \| \bH \| _{\mathcal{L}_2}:=\left( \int_{-\infty}^{\infty} \| \bH(\imath \omega ) \| _F ^2 \ \mathrm{d} \omega \right)^{\frac{1}{2}} < \infty \right \}, \\
\mathcal{L}_\infty(\imath \mathbb R)&:= \left\{ \bH\colon \mathbb C \to \mathbb C^{p \times m}\ | \ \bH \text{ is meromorphic  }, \| \bH \| 
_{\mathcal{L}_\infty}:=\sup_{\omega \in \mathbb R} \sigma_{\mathrm{max}}( \bH(\imath \omega ) ) < \infty \right \}.
\end{align*} 
}
In particular, we have the orthogonal decomposition $\mathcal{L}_2(\imath \mathbb R) = \mathcal{H}_2 (\mathbb C^-) \oplus \mathcal{H}_2(\mathbb C^+).$
 %\texttt{Do we have to introduce/define 
%$\mathcal{L}_\infty(\imath \mathbb R)$ as well?}
%\todosg[inline]{I think we should introduce $\mathcal{L}_\infty(\imath \mathbb R)$ as well. We are using it below. Indeed we should introduce $\mathcal{H}_\infty$ and $\mathcal{H}_2$  as well before they are used in \eqref{eq:gap_metric_error} and \eqref{eq:H2_gap} above.}

Next, we prove a similar result as in \cite[Lemma 2.2]{Par91}.
%The following bound is based on a similar SISO result shown in \cite[Lemma 2.2]{Par91}. 
%\todosg[inline]{Tobi, should we label this as``based on a similar result" or as " we prove a similar result" or " we extend the results of .. "?}
\begin{proposition}\label{prop:l2_bound}
  Let $(\bA,\bB,\bC)$ and $(\whA,\whB,\whC)$ be two stabilizable and 
detectable LTI systems with $\sigma(\bA)\cap \imath \mathbb R =\emptyset = 
\sigma(\whA) \cap \imath \mathbb R.$ Let further $\bG_\bF=[\bM,\bN], 
\whG_\bF=[\whM,\whN]$ be left-coprime factorizations with $\bF=\bP \bC^T$ and 
$\whF=\whP \whC^T.$ Then,
\begin{align} \label{eq:GminusGhatl2}
  \| \bG-\whG\| _{\mathcal{L}_2} &\le \| \whM^{-1}\|  
_{\mathcal{L}_\infty} \left( \| \bG\| _{\mathcal{L}_\infty} \| \bM -\whM 
\|_{\mathcal{H}_2}    + \|\bN-\whN\|_{\mathcal{H}_2} \right),
\end{align}
and, thus
\begin{align} \label{eq:GminusGhatl2viagap}
  \| \bG-\whG\| _{\mathcal{L}_2} &\le \| \whM^{-1}\|  _{\mathcal{L}_\infty} \left( 1+ \| \bG\| _{\mathcal{L}_\infty}  \right)\left\| \bG-\whG \right\|_{\mathcal{H}_2\text{-gap}}.
\end{align}
\end{proposition}
\begin{proof}
  By using the left-coprime factorizations of $\bG$ and $\whG$, we obtain  
  \begin{align*}
    \bG-\whG &= \bG-\whM^{-1}\whN 
    = \whM^{-1} (\whM \bG - \whN ) \\
    &= \whM^{-1} (\whM \bG - \bM \bM^{-1} \bN + \bN - \whN ) \\
    &= \whM^{-1} \left( (\whM-\bM) \bG + (\bN -\whN) \right).
  \end{align*}
Since $\bA$ and $\whA$ are assumed to have no purely imaginary eigenvalues,  we know that $\bG \in \mathcal{L}_\infty(\imath \mathbb R)$ and $\whM^{-1}(\cdot) = \bI + \whC(\cdot\bI-\whA)^{-1}\whF \in \mathcal{L}_\infty(\imath \mathbb R).$ The  assertion \eqref{eq:GminusGhatl2} then follows from the fact that $\| \bG \bH \| _{\mathcal{L}_2} \le \| \bG \| _{\mathcal{L}_\infty} \| \bH \| _{\mathcal{L}_2}$ for all $\bG \in \mathcal{L}_\infty$ and $\bH \in \mathcal{L}_2.$ In particular, note that $\bM-\whM \in \mathcal{H}_2$ and $\bN-\whN\in \mathcal{H}_2$, which implies $  \| \bM-\whM \|_{\mathcal{L}_2} = \| \bM -\whM \|_{\mathcal{H}_2},  \  \| \bN-\whN \|_{\mathcal{L}_2} = \| \bN -\whN \|_{\mathcal{H}_2}. $  Finally, the assertion \eqref{eq:GminusGhatl2viagap} directly follows from \eqref{eq:GminusGhatl2} and the definition of the $\mathcal{H}_2$-gap in \eqref{eq:H2_gap}.
\end{proof}

\begin{remark}
  Note that if we split $\bG=\bG_s+\bG_u$ and $\whG=\whG_s+\whG_u$ into their 
stable and unstable parts and use the orthogonal decomposition $\mathcal{L}_2(\imath 
\mathbb R) = \mathcal{H}_2 (\mathbb C_-) \oplus 
\mathcal{H}_2(\mathbb C_+),$ we also obtain bounds for 
$\|\bG_s-\whG_s\|_{\mathcal{H}_2(\mathbb C_+)}$ and 
$\|\bG_u-\whG_u\|_{\mathcal{H}_2(\mathbb C_-)}$: 
\begin{align*}
\max\left(\|\bG_s-\whG_s\|_{\mathcal{H}_2(\mathbb C_+)},\|\bG_u-\whG_u\|_{\mathcal{H}_2(\mathbb C_-)}\right) \le \| \whM^{-1}\|  _{\mathcal{L}_\infty} \left( 1+ \| \bG\| _{\mathcal{L}_\infty}  \right)\left\| \bG-\whG \right\|_{\mathcal{H}_2\text{-gap}}.
\end{align*}
\end{remark}
%\todosg[inline]{Should we write them out explicitly?}
Proposition \ref{prop:l2_bound}  bounds the $\mathcal{L}_2$ distance between the full model 
$\bG(s)$ and the reduced model $\whG(s)$ with the $\mathcal{H}_2$\text{-gap} distance between
full closed-loop system $\bG_\bF(s)$ and the reduced one $\whG_\bF(s)$. This immediately motivates a model reduction approach  in which one tries to minimize the $\mathcal{H}_2$\text{-gap}. This is what we investigate next.

\section{$\mathcal{H}_2\text{-gap}$ model reduction}

%\todosg[inline]{Note to self: In this section, we achieve/show ....}

In this section, we analyze the $\mathcal{H}_2$\text{-gap} in more detail. We begin with the derivation of a pole-residue formula that extends the one discussed for the standard case in, e.g., \cite{GugAB08}. Subsequently, we discuss the individual error terms from a rational interpolation-based perspective which suggests the use of an iterative algorithm generalizing IRKA.

\subsection{Pole-residue formulae for the $\mathcal{H}_2\text{-gap}$ measure}
Recall the state-space representation of $\bG_\bF(s)$:
\begin{align}
\bG_\bF(s)  & =  [\bI, 0] +  \bC (s\bI -\bA_\bF)^{-1} [-\bF,~\bB]  \tag{\ref{eqn:statespaceG}}.
\end{align}
Let  $\bw_i$  denote the left eigenvector of 
$\bA_\bF$ associated with the eigenvalue $\lambda_i$, i.e., $\bw_i^T \bA_\bF = \lambda_i \bw_i^T$\footnote{We define
the left eigenpair via the relationship $\bw_i^T \bA_\bF = \lambda_i \bw_i^T$ as opposed to the more usual definition: $\bw_i^* \bA_\bF = \lambda_i \bw_i^*$. Even though $\bw_i$ is potentially a complex vector,  the version we adopt eliminates the need to use $\overline{\bw}_i$ in many definitions and equations that follow.}. For simplicity of presentation, assume the poles  $\lambda_i$  are semi-simple and write
$\bW^T \bA_\bF = \bLamb \bW^T$ where $\bW = [\bw_1,\,\bw_2,\,\ldots,\,\bw_n] \in \IC^{n\times n}$ 
and $\bLamb = \textsf{diag}(\lambda_1,\,\lambda_2,\,\ldots,\lambda_n)\in \IC^{n\times n}$.  Define $\bV = \bW^{-T}$.
Then, a state-space transformation by $\bW^T$ in \eqref{eqn:statespaceG} yields the pole-residue representation of $\bG_\bF(s)$: 
\begin{align} \label{eqn:GfGpr}
 \bG_\bF(s) &=  [\bI,0]+\sum_{i=1}^n 
\frac{\bc_i [\bff_i^T,\bb_i^T]}{s-\lambda_i},
\end{align}
where
\begin{equation}\label{eq:res_dir_GF}
  \bb_i =\bB^T \bw_i, \quad   \bff_i =-\bF^T\bw_i=-\bC \bP \bw_i,\quad\mbox{and}\quad\bc_i = \bC \bv_i.
\end{equation}
We follow the same line of arguments to obtain the pole-residue representation of $\whG_\bF$. 
Let $\whw_j$ be the left eigenvector of 
$\whA_\bF$ associated with the eigenvalue $\widehat{\lambda}_j$. Assume the poles  $\widehat{\lambda}_j$  are semi-simple, and define
$\whW = [\whw_1,\,\whw_2,\,\ldots,\,\whw_r] \in \IC^{r\times r}$  and $\whV = \whW^{-T}$. Then,
\begin{align} \label{eqn:GfGfhatpr}
 \whG_\bF(s)  =[\bI,0]+\sum_{j=1}^r 
\frac{\whc_j [\whf_j^T,\whb_j^T]}{s-\widehat{\lambda}_j},
\end{align}
where
\begin{equation}\label{eq:res_dir}
  \whb_j =\whB^T \whw_j, \quad   \whf_j=-\whF^T\whw_j =-\whC \whP \whw_j,\quad\mbox{and}\quad\whc_j = \whC \whv_j.
\end{equation}
Now, using the  representations \eqref{eqn:GfGpr} and \eqref{eqn:GfGfhatpr}, we extend the pole-residue based formula for the $\mathcal{H}_2$-norm to the $\mathcal{H}_2\text{-gap}$ norm.
\begin{proposition}\label{prop:cph2_res}
Let $\bG$ and $\whG$ be stabilizable detectable with  coprime factorizations in \eqref{eq:Gcoprime} and \eqref{eq:Ghatcoprime}. Further, let $\bG_\bF$ and $\whG_\bF$ have the pole-residue representations in \eqref{eqn:GfGpr} and \eqref{eqn:GfGfhatpr}. Then,
%  \begin{align}  
% \|\bG-\whG \|_{\mathcal{H}_2\text{-gap}}^2 &= \sum_{i=1}^n 
%\bc_i^T(\bG_\bF(-\lambda_i) 
%-\whG_\bF(-\lambda_i) ) \begin{bmatrix} \bff_i \\ \bb_i \end{bmatrix} \\ 
%%%
%&\qquad + \sum_{j=1}^r \whc_j^T 
%(\whG_\bF(-\widehat{\lambda}_{j}) - \bG_\bF (-\widehat{\lambda}_{j}) ) 
%\begin{bmatrix} \whf_j \\ \whb_j  \end{bmatrix}.
%\end{align}
  \begin{align}   \label{eq:H2gappr}
  \begin{array}{c}
  {\displaystyle 
 \|\bG-\whG \|_{\mathcal{H}_2\text{-gap}}^2 = \sum_{i=1}^n 
\bc_i^T(\bG_\bF(-\lambda_i) 
-\whG_\bF(-\lambda_i) ) \begin{bmatrix} \bff_i \\ \bb_i \end{bmatrix} }
 \\ 
  {\displaystyle   \qquad\hspace{1.5in} + \sum_{j=1}^r \whc_j^T 
(\whG_\bF(-\widehat{\lambda}_{j}) - \bG_\bF (-\widehat{\lambda}_{j}) ) 
\begin{bmatrix} \whf_j \\ \whb_j  \end{bmatrix}.}
\end{array}
\end{align}
\end{proposition}
\begin{proof}
  First, we rewrite the error  as
  \begin{align*}
    \|\bG-\whG \|_{\mathcal{H}_2\text{-gap}}^2 = \| 
\bG_\bF-\whG_\bF\|_{\mathcal{H}_2}^2 = \| (\bG_\bF-[\bI,0])+([\bI,0]-\whG_\bF)  
\|_{\mathcal{H}_2}^2.
  \end{align*}
Note that $(\bG_\bF-[\bI,0]) \in \mathcal{H}_2$ and $([\bI,0]-\whG_\bF) \in 
\mathcal{H}_2$. For a transfer function $\bH \in \mathcal{H}_2$ with pole-residue representation $\bH(s) = \sum_{k=1}^n \frac{\bh_k \bg_k^T}{s-\mu_i}$, the $\mathcal{H}_2$ norm satisfies
$
\| \bH  \|_{{\mathcal{H}_2}}^2 = \sum_{k=1}^n\bh_k^T\bH(-\mu_k) \bg_k; 
$
see, e.g., \cite[Lemma 
2.4]{GugAB08} for the SISO version and \cite[Lemma 1.1]{AntBG10} for the MIMO version.
Then, the  result \eqref{eq:H2gappr} follows from applying this  $\mathcal{H}_2$ norm formula to the pole-residue representation
of $\bG_\bF-\whG_\bF$, which can be obtained from
   \eqref{eqn:GfGpr} and \eqref{eqn:GfGfhatpr} by eliminating the leading (constant) terms.
 \end{proof}

\subsection{$\mathcal{H}_2\text{-gap}$ formula and interpolation.}
%\todo[inline]{In this section, we first eliminate $G_F$ from the formula and write everything in terms $G$ and $\hat{G}$. This will then allow us to eliminate parts of the error terms without ever needing $P$.}

Proposition \ref{prop:cph2_res} reveals two components that contribute to the $\htwogap$, in a way that is similar to the standard $\mathcal{H}_2$-error measure,  The first component is due to the mismatch of the transfer functions  $\bG_\bF$ and $\whG_\bF$ at the mirror images of the \emph{full-order closed-loop poles}, $\lambda_i$, and the second component reflects the mismatch at the mirror images of the \emph{reduced-order closed-loop poles}, $\widehat{\lambda}_i$. In order to reduce the $\htwogap$, a reasonable approach might be to eliminate terms from these two components. For example, if one enforces $(\bG_\bF(-\lambda_i)  -\whG_\bF(-\lambda_i) ) \begin{bmatrix} \bff_i \\ \bb_i \end{bmatrix}  = 0$, then the $i$th term from the first term will be eliminated. This condition is referred to as \emph{right-tangential interpolation}; more specifically, we state $\whG_\bF(s)$ tangentially interpolates $\bG_\bF(s)$ at the interpolation point $-\lambda_i$ along the right-tangential direction $ \begin{bmatrix} \bff_i \\ \bb_i \end{bmatrix} $; see \cite{AntBG10,AntBG20} for further details. We can  eliminate the $i$th term in the first sum by enforcing $\bc_i^T(\bG_\bF(-\lambda_i)-\whG_\bF(-\lambda_i)) = 0$ as well. This is referred to as left-tangential interpolation. The terms in the second sum can be similarly eliminated. This interpretation of the  $\cH_2$-error norm and elimination of the error terms via interpolation have been proposed in the regular $\cH_2$-error measure \cite{Gug03}. Since the second-term depends on the reduced-model to-be-computed and are not known a priori, \cite{Gug03} proposed eliminating the dominant terms from the first term. Even though this is not an optimal reduction strategy, this approach has worked well in various examples. The situation is rather different here. 

In order to minimize the $\htwogap$, we construct a reduced model $\whG$ from $\bG$.   Yet Proposition \ref{prop:cph2_res} shows that the error depends on $\bG_\bF$ and $\whG_\bF$, which we do not have direct access to. Consider the $\htwogap$ once again: $\|\bG  - \whG\|_{\htwogap} =\|\bG_\bF  - \whG_\bF\|_{\htwo}.$ In order to minimize the $\htwogap$, suppose we perform an optimal $\htwo$ reduction on $\bG_\bF$, and let $\bV \in \mathbb{R}^{n \times r}$ and $\bW\in \mathbb{R}^{n \times r}$ be the corresponding optimal model reduction bases with $\bW^T\bV = \bI$. The state-space representation for the reduced $\whG$ is given by
\begin{align*}
\whG_\bF(s) & = 
  [\bI, 0] +  \whC (s\bI -\whA_\bF)^{-1} [-\whF,~\whB]  %\\ &=  [\bI, 0] +  \bC\bV (s\bI -\bW^T\bA_\bF\bV)^{-1} [-\bW^T\bF,~\bW^T\bB] 
  \\ &=  [\bI, 0] +  \bC\bV (s\bI -\bW^T\bA \bV - \bW^T\bF\bC\bV)^{-1} [-\bW^T\bF,~\bW^T\bB] .
\end{align*}
We need to extract the corresponding reduced model $\whG$ from this reduced closed-loop model $\whG_\bF(s)$. One might reasonably assume that $\whC = \bC \bV $, $\whB = \bW^T\bB$, $\whF = \bW^T\bF$, and $\whA = \bW^T \bA \bV$. However,  these reduced quantities need also to satisfy $\whF = \bW^T \bF = \bW^T \bP \bC = \whP \whC^T$ where $ \whP$ solves
$
  \whA \whP + \whP \whA^T - \whP \whC^T \whC\whP + \whB \whB^T =0.
$
Clearly, this is not true in general and we cannot expect to extract a reduced system $\whG$ that would have created the reduced closed-loop model $\whG_\bF(s)$. A similar issue arises in the weighted $\htwo$ model reduction problem, where, given a weighting functions $\bW_o(s)$,  one tries to minimize the weighted error $\| \bW_o (\bG - \whG ) \|_{\htwo}$. As \cite{AniBGA13} and \cite{BreBG15} prove, in the weighted-$\htwo$ problem, the error (and optimality) requires that a function of  $\bG$ interpolates a function of  $\whG$, leading to the same issue that we encounter here. 

To address this issue at least partially, Lemmas \ref{lem:res_and_G}--\ref{lem:G_vs_GP2} enable us to rewrite the $\htwogap$ formula in \eqref{eq:H2gappr} as a function of $\bG$ and $\whG$. The result  in Theorem \ref{thm:main_result} will, then, form the foundation of the proposed method outlined in Algorithm \ref{alg:gap_irka}.

\begin{lemma}\label{lem:res_and_G}
Let $\bG$  be a stabilizable and detectable linear system with  coprime factorization in \eqref{eq:Gcoprime}. Further, let $\bG_\bF = [\bM~\bN]$  have the pole-residue representation 
$\bG_\bF(s)=[\bI,0]+\sum_{i=1}^n \frac{\bc_i [\bff_i^T,\bb_i^T]}{s-\lambda_i}$ as  in \eqref{eqn:GfGpr}. 
%Let $(\bA,\bB,\bC)$ denote a stabilizable (and detectable ???) linear system
%% in 
%% pole residue form  
%with 
%$\bG_\bF(s)=[\bI,0]+\sum_{i=1}^n \frac{\bc_i 
%[\bff_i^T,\bb_i^T]}{s-\lambda_i}$
%% and $ \whG_P(s) =\sum_{j=1}^r 
%% \frac{\whc_j 
%% [\whb_{1,j}^T,\whb_{2,j}^T]}{s-\widehat{\lambda}_j}$ 
%% be given
%in pole  residue form. 
Assume that $\sigma(\bA)\cap \sigma(-\bA_\bF) = \emptyset.$
% = 
% \sigma(\whA)\cap \sigma(-\whA_P).$
Then, 
% and $j=1,\dots,r$ 
 \begin{align} \label{eq:Gandf}
  \bG(-\lambda_i)\bb_i &= - \bff_i \qquad\mbox{and}\qquad
   \bG_\bF(-\lambda_i) \begin{bmatrix} \bff_i 
\\ \bb_i  \end{bmatrix}=0,~~\mbox{for}~~i=1,\dots,n.
 \end{align}
\end{lemma}
\begin{proof}
Let $\bw_i$  denote the left eigenvector of 
$\bA_\bF$ associated with the eigenvalue $\lambda_i$, i.e., $\bw_i^T \bA_\bF = \lambda_i \bw_i^T$.
 Definition of $\bb_i$ in \eqref{eq:res_dir_GF} yields
 \begin{align}\label{eq:aux1}
   \bG(-\lambda_i)\bb_i = 
%    \lim_{s\to \lambda_i}(s-\lambda_i) 
\bC(-\lambda_i \bI -\bA)^{-1} \bB \bB^T\bw_i.
 \end{align}
% Since $\bw_i$ is a left eigenvector of $\bfA_P,$ we conclude that
% \begin{align*}
%     (s\bfI- \bfA_P^T) \bw_i = 
% (s-\lambda_i) \bw_i 
% \end{align*}
% and, as a consequence, we obtain
% \begin{align}\label{eq:aux1}
% \bfG(-\lambda_i)\bfb_{1,i} = \lim_{s\to \lambda_i}(s-\lambda_i) 
% \bfC(-\lambda_i \bfI -\bfA)^{-1} \bfB \bfB^T \frac{\bw_i}{s-\lambda_i}.
% \end{align}
Using the fact that $\bA_\bF = \bA - \bP\bC^T\bC$, the Riccati equation \eqref{eq:Riccati_o} can be rewritten as the  Sylvester equation 
$
 \bA \bP + \bP \bA_\bF^T + \bB\bB^T=0.
$
Postmultiplication with $\bw_i$ then yields 
\begin{align*}
  \bP \bw_i = (-\lambda_i \bI -\bA)^{-1} \bB \bB^T \bw_i.
\end{align*}
Inserting this last expression into \eqref{eq:aux1} leads to 
\begin{align*}
 \bG(-\lambda_i)\bb_i=\bC \bP \bw_i=-\bff_i,
 \end{align*}
 which proves the first assertion in \eqref{eq:Gandf}.
% The assertion for the reduced-order model follows by the same arguments. 
%\end{proof}
%
%
%\begin{lemma}\label{lem:G_vs_GP1}
%Let $(\bA,\bB,\bC)$ denote a stabilizable and detectable linear system  
%with 
%$\bG_\bF(s)=[\bI,0]+\sum_{i=1}^n \frac{\bc_i 
%[\bff_i^T,\bb_i^T]}{s-\lambda_i}$ 
%in pole  residue form. 
%Assume that $\sigma(\bA)\cap \sigma(-\bA_\bF) = \emptyset.$ 
%Then, for all $i=1,\dots,n$ 
%it holds that $  \bG_\bF(-\lambda_i) \begin{bmatrix} \bff_i 
%\\ \bb_i  \end{bmatrix}=0.$
%\end{lemma}
%\begin{proof}
To prove the second assertion, we compute
\begin{align*}
 \bG_\bF(-\lambda_i) \begin{bmatrix} \bff_i 
\\ \bb_i \end{bmatrix} = 
[ \bM(-\lambda_i)~ \bN(-\lambda_i)] \begin{bmatrix} \bff_i 
\\ \bb_i \end{bmatrix}
 &= \bM(-\lambda_i) \bff_i + \bN(-\lambda_i) 
\bb_i   \\
&= \bM(-\lambda_i)( \bff_i + [\bM(-\lambda_i)]^{-1}\bN(-\lambda_i) \bb_i )  \\
&= \bM(-\lambda_i)(\bff_i + \bG(-\lambda_i) \bb_i) = 0,
\end{align*}
where, in the last step we used the just-proven first assertion in 
\eqref{eq:Gandf}.
\end{proof}

\begin{lemma}\label{lem:G_vs_GP2}
Let $\bG$  be a stabilizable and detectable linear system with the closed-loop transfer function $\bG_\bF(s)=[\bM~\bN]= [\bI,0]+\sum_{i=1}^n \frac{\bc_i [\bff_i^T,\bb_i^T]}{s-\lambda_i}$
as  in \eqref{eqn:GfGpr}. Let $\whG_\bF=[\whM~\whN]$ denote the closed-loop transfer function of a stabilizable and detectable system reduced model $\whG$. If $\sigma(\bA)\cap \ \sigma(-\bA_\bF) = \emptyset$, then
  \begin{align} \label{eq:MhatGGhat}
    \whG_\bF(-\lambda_i)\begin{bmatrix}\bff_i \\ \bb_i \end{bmatrix}&=  -
\whM(-\lambda_i) \left( \bG(-\lambda_i) - \whG(-\lambda_i) \right)\bb_i,
~~\mbox{for}~~
i=1,\dots,n.
  \end{align}
\end{lemma}
\begin{proof}
Using $\whG(s) = [\whM(s)]^{-1}\whN(s)$, we evaluate 
\begin{align*}
 \whG_\bF(-\lambda_i) \begin{bmatrix} \bff_i 
\\ \bb_i \end{bmatrix}  &= \whM(-\lambda_i) \bff_i + \whN(-\lambda_i) 
\bb_i   \\
&= \whM(-\lambda_i)( \bff_i + [\whM(-\lambda_i)]^{-1}\whN(-\lambda_i) \bb_i )  
\\
&= \whM(-\lambda_i)(-\bG(-\lambda_i)\bb_i+ \whG(-\lambda_i) \bb_i),
\end{align*}
where in the last step we used the first assertion in 
\eqref{eq:Gandf}.
\end{proof}

As a direct consequence of Proposition \ref{prop:cph2_res}, and Lemmas \ref{lem:res_and_G} and  \ref{lem:G_vs_GP2}, we obtain an alternative representation of the $\mathcal{H}_2$-gap error, one of our main results.

\begin{theorem} \label{thm:main_result}
Let $\bG$ and $\whG$ be stabilizable and detectable with  coprime factorizations in \eqref{eq:Gcoprime} and \eqref{eq:Ghatcoprime}. Further, let $\bG_\bF$ and $\whG_\bF$ have the pole-residue representations in \eqref{eqn:GfGpr} and \eqref{eqn:GfGfhatpr}. Then,
%  \begin{align}
%    \|\bG-\whG \|_{\mathcal{H}_2\text{-gap}}^2&= \sum_{i=1}^n \bc _i^T 
%\whM(-\lambda_i) \left( \bG(-\lambda_i) - \whG(-\lambda_i) \right)\bb_i \\
%%%
%&\qquad + \sum_{j=1}^r \whc_j^T \bM(-\widehat{\lambda}_j) \left( \whG(- 
%\widehat{\lambda}_j) - \bG (-\widehat{\lambda}_j) \right) \whb_j.
%  \end{align}
    \begin{align} \label{eq:h2gapGGhat}
    \begin{array}{c}
   {\displaystyle 
    \|\bG-\whG \|_{\mathcal{H}_2\text{-gap}}^2= \sum_{i=1}^n \bc _i^T 
\whM(-\lambda_i) \left( \bG(-\lambda_i) - \whG(-\lambda_i) \right)\bb_i } \\
{\displaystyle 
\qquad \hspace{1.1in}+ \sum_{j=1}^r \whc_j^T \bM(-\widehat{\lambda}_j) \left( \whG(- 
\widehat{\lambda}_j) - \bG (-\widehat{\lambda}_j) \right) \whb_j.}
\end{array}
  \end{align}
\end{theorem}
\begin{proof}
Consider the first sum in the $\htwogap$ formula in \eqref{eq:H2gappr}, i.e., $\sum_{i=1}^n \bc_i^T(\bG_\bF(-\lambda_i) -\whG_\bF(-\lambda_i) ) \begin{bmatrix} \bff_i \\ \bb_i \end{bmatrix}$. First using the second assertion of  Lemma \ref{lem:res_and_G} and then using Lemma \ref{lem:G_vs_GP2} yield
{\small
$$
\sum_{i=1}^n 
\bc_i^T(\bG_\bF(-\lambda_i) 
-\whG_\bF(-\lambda_i) ) \begin{bmatrix} \bff_i \\ \bb_i \end{bmatrix}
 = - \sum_{i=1}^n 
\bc_i^T \whG_\bF(-\lambda_i) \begin{bmatrix} \bff_i \\ \bb_i \end{bmatrix}=  \sum_{i=1}^n \bc _i^T 
\whM(-\lambda_i) \left( \bG(-\lambda_i) - \whG(-\lambda_i) \right)\bb_i, 
$$}
which is the first sum, indexed by $i$, in \eqref{eq:h2gapGGhat}. The second part of   \eqref{eq:h2gapGGhat} follows similarly by interchanging the roles of 
$\bG_\bF$ and $\whG_\bF$ in Lemmas 
\ref{lem:res_and_G} and \ref{lem:G_vs_GP2}.
\end{proof}
Theorem \ref{thm:main_result} achieves what we were set to accomplish; representing the $\htwogap$ in terms of the model to be reduced, $\bG$, and 
the reduced model itself, $\whG$. Now, we can reduce $\bG$, e.g., via interpolatory projection-based methods, to construct the reduced model $\whG$ that tangentially interpolates $\bG$ and then we eliminate the selected terms from the error formula in \eqref{eq:h2gapGGhat}. We will make this interpolation aspect more concrete next. 

The following result is a direct consequence of Lemmas \ref{lem:res_and_G} and \ref{lem:G_vs_GP2}.

\begin{corollary} \label{cor:Ghat_to_GFhat}
Assume the same set-up in Lemmas  \ref{lem:res_and_G} and \ref{lem:G_vs_GP2}.
\begin{align}
\mbox{If}~~\whG(-\lambda_i)\bb_i=\bG(-\lambda_i)\bb_i,~~\mbox{then} 
~~~\whG_\bF(-\lambda_i) \begin{bmatrix} \bff_i \\ \bb_i 
\end{bmatrix}=\bG_\bF(-\lambda_i) 
\begin{bmatrix} \bff_i \\ \bb_i \end{bmatrix} = 0~~\mbox{for}~~
i=1,\dots,n.
\end{align} 
Moreover,
\begin{align}
\mbox{If}~~\whG(-\widehat{\lambda}_j)\whb_j=\bG(-\widehat{\lambda}_j)\whb_j,~~\mbox{then} ~~~\whG_\bF(-\widehat{\lambda}_j) \begin{bmatrix} \whf_j \\ \whb_j 
\end{bmatrix}=\bG_\bF(-\widehat{\lambda}_j) 
\begin{bmatrix} \whf_j \\ \whb_j \end{bmatrix}=0,~~\mbox{for}~~
j=1,\dots,r.
\end{align} 
%
%  If $ \whG(-\lambda_i)\bb_i=\bG(-\lambda_i)\bb_i $ then 
%$\whG_\bF(-\lambda_i) \begin{bmatrix} \bff_i \\ \bb_i 
%\end{bmatrix}=\bG_\bF(-\lambda_i) 
%\begin{bmatrix} \bff_i \\ \bb_i \end{bmatrix}.$ 
\end{corollary} 
%
%\begin{corollary}
% If $\whG(-\widehat{\lambda}_j)\whb_j=\bG(-\widehat{\lambda}_j)\whb_j $ 
%then 
%$\whG_\bF(-\widehat{\lambda}_j) \begin{bmatrix} \whf_j \\ \whb_j 
%\end{bmatrix}=\bG_\bF(-\widehat{\lambda}_j) 
%\begin{bmatrix} \whf_j \\ \whb_j \end{bmatrix}.$ 
%\end{corollary}
Corollary \ref{cor:Ghat_to_GFhat} first reveals that we can enforce the closed-loop systems $\bG_\bF$ and $\whG_\bF$ to tangentially interpolate each other by forcing interpolation of $\bG$ and $\whG$. The resulting interpolation is occurring at specially adapted points, namely at the mirror images of the full- or reduced-order closed-loop poles. Moreover, the interpolated value is zero.  It is worth mentioning that reduced-order closed-loop poles have also been studied in the context of rational Krylov subspace methods for solving the algebraic Riccati equation in \cite{LinS15}. In that work, the authors showed that the rational Krylov subspace method coincides with a subspace iteration if the shifts are chosen as the mirrored reduced-order closed-loop poles.

Corollary \ref{cor:Ghat_to_GFhat} also reveals  that we can enforce interpolation of the closed-loop model $\bG_\bF$ \emph{without ever constructing 
$\bG_\bF$}, i.e., without needing to solve the (large-scale) Riccati equation 
\eqref{eq:Riccati_o} to compute $\bP$. This will have substantial   numerical advantages in the large-scale settings, because unlike most methods used for model reduction using the gap measure, one does not need to solve for the Gramian $\bP$.  Next, we will discuss the numerical framework to enforce these desired interpolation conditions.

\subsection{Model reduction with respect to the $\htwogap$ measure} 

First, we review briefly the projection-based tangential interpolation framework. For details, we refer the reader to, e.g., \cite{GalVVD05,AntBG10,BeaG17,AntBG20}. 
Let $\bG(s) = \bC(s\bI-\bA)^{-1}\bB$ denote the transfer function of the full-model with $m$-inputs and $p$-outputs. Suppose the left-interpolation points $\{\mu_1,\mu_2,\ldots,\mu_r\}\in \IC$ are chosen together with
non-trivial left-directions $\{\Ldir_1,\Ldir_2,\ldots,\Ldir_r \} \in \IC^{p}$. Also suppose 
the right-interpolation points $\{\sigma_1,\sigma_2,\ldots,\sigma_r\}\in \IC$ are chosen together with
non-trivial right-directions $\{\Rdir_1,\Rdir_2,\ldots,\Rdir_r \}\in \IC^{m}$. Construct the model reduction bases $\bV_r\in \IC^{n\times r}$ and $\bW_r \in \IC^{n\times r}$:
\begin{align*}
\bV_r&=\left[(\sigma_1\bI-\bA)^{-1}\bB\Rdir_1, (\sigma_2\bI-\bA)^{-1}\bB\Rdir_2, \dots,(\sigma_r\bI-\bA)^{-1}\bB \Rdir_r\right]~~\mbox{and}\\
\bW_r&=\left[(\mu_1\bI-\bA^T)^{-1}\bC^T \Ldir_1,
(\mu_2\bI-\bA^T)^{-1}\bC^T \Ldir_2, 
\dots,(\mu_r\bI-\bA^T)^{-1}\bC^T \Ldir_r\right].
\end{align*}
Assume, without loss of generality that a basis transformation is performed and $\bW_r^T\bV_r= \bI_r$. Construct the reduced model $\whG(s) = \whC(s\bI-\whA)^{-1}\whB$ via Petrov-Galerkin projection, i.e.,
$$
\whA = \bW_r^T \bA\bV_r,~~\whB = \bW_r^T\bB,~~\mbox{and}~~\whC=\bC \bV_r.
$$
Then, the reduced model $\whG$ tangentially interpolates $\bG$ in the sense that
$$
\bG(\sigma_j) \Rdir_j = \whG(\sigma_j) \Rdir_j,~~
\Ldir_j^T\bG(\mu_j) =  \Ldir_j^T\whG(\mu_j),
~~ \Ldir_j^T\bG(\sigma_j) \Rdir_j = \Ldir_j^T\whG(\sigma_j) \Rdir_j,
$$
for $j=1,2,\ldots,r$. Moreover, if $\sigma_k = \mu_k$,  then one additionally satisfies a tangential Hermite interpolation, namely $ \Ldir_k^T\bG'(\sigma_k) \Rdir_k = \Ldir_k^T\whG'(\sigma_k) \Rdir_k$
where `` $'$ " denotes the derivate with respect to $s$. Therefore, if the interpolation points and directions are specified, then constructing a reduced interpolatory transfer function can be easily constructed as described, with the main cost of solving the shifted linear systems in computing $\bV$ and $\bW$. Then,
the natural question to ask is how to choose the interpolation points and directions to minimize an error measure. This question has been answered using the regular $\htwo$ error measure. Let $\displaystyle\whG(s)= \sum_{j=1}^r \frac{\Ldir_j\Rdir_j^T}{s + \sigma_j}$ be the pole-residue decomposition. If $\whG(s)$ is the $\htwo$-optimal approximation to $\bG(s)$ in the $\htwo$  norm, then, 
$\bG(\sigma_j) \Rdir_j = \whG(\sigma_j) \Rdir_j$, 
$\Ldir_j^T\bG(\sigma_j) =  \Ldir_j^T\whG(\sigma_j)$, and $\Ldir_j^T\bG'(\sigma_j) \Rdir_j = \Ldir_j^T\whG'(\sigma_j) \Rdir_j$, for $j=1,2,\ldots,r$. Therefore, tangential Hermite interpolation is a necessary condition for $\htwo$ optimality. Note that optimal interpolation points are $\{\sigma_j\}$, the mirror images of the poles of the reduced model $\whG(s)$,  and the optimal tangential directions are based on the residues of $\whG(s)$; neither known a priori. The Iterative Rational Krylov Algorithm (IRKA) of \cite{GugAB08} and its variants \cite{Bunetal10,BeaG12,beattie2007kbm,xu2010optimal} resolve this issue by iteratively correcting the interpolation points and directions until the desired optimality conditions are met. For details, we refer the reader to \cite{GugAB08,AntBG10,BeaG17,AntBG20} and the references therein. 

The situation is similar in the $\htwogap$ problem we consider here. First, leave the question of optimality aside and focus on reasonable/well-informed interpolation points and directions selection. Recall  the $\htwogap$ error formula in \eqref{eq:h2gapGGhat}. We can eliminate the $i$th term in the first sum by choosing $\sigma_i = -\lambda_i$ as an interpolation point and $\Rdir_i = \bb_i$ as an interpolation direction. Clearly, the first sum has $n$ components and one can only eliminate $r$ conditions from there using $r$ interpolation points. One can choose the poles $\lambda_i$ with \emph{dominant} residue terms $\bc_i\bb_i^T$, for example. However, this requires that we compute the full-order closed-loop poles $\lambda_i$ by solving for the Gramian $\bP$. Also, as discussed above, the regular $\htwo$ minimization via interpolation reveals that the optimal interpolation points are determined by the reduced-order poles, not the full-order ones.

To eliminate the $j$th term from the second sum in \eqref{eq:h2gapGGhat}, we  can enforce $ \whG(-\widehat{\lambda}_j) \whb_j =  \bG (-\widehat{\lambda}_j)  \whb_j$. This puts us in the framework of the regular $\htwo$ problem. The interpolation points $\sigma_j = -  \widehat{\lambda}_j$ and the tangental directions $\Rdir_j = \whb_j$, for $j=1,2,\ldots,r$, depend on the reduced-model $\whG$ (or more precisely $\whG_\bF$) we want to compute; thus the interpolation data is not known a priori. Thus, as in IRKA, this requires an iterative algorithm that adaptively corrects the interpolation data. The major advantage compared to eliminating terms from the first sum in \eqref{eq:h2gapGGhat} is that this adaptive correction process does not require computing full-order closed-loop poles, i.e., computing $\bP$. Yet, as Corollary \ref{cor:Ghat_to_GFhat} illustrates, we are still able to interpolate the full-order closed-loop model.

Algorithm \ref{alg:gap_irka} gives a sketch of the proposed numerical scheme. 
Starting with an initial selection of interpolation data, the algorithm computes an interpolatory reduced model $\whG$ (Lines 2--4) and then computes the pole-residue representation of the reduced-order closed-loop model $\whG$ (Lines 5--6). Note that these computations are trivial since it is performed at the reduced-order dimension. Line 7 updates the interpolation data so that upon convergence of Algorithm \ref{alg:gap_irka}, we have $\sigma_j = -  \widehat{\lambda}_j$ and $\Rdir_j = \whb_j$, for $j=1,2,\ldots,r$, as we wanted to accomplish.
Upon convergence of Algorithm \ref{alg:gap_irka}, we enforce $ \whG(-\widehat{\lambda}_j) \whb_j =  \bG (-\widehat{\lambda}_j)  \whb_j$ and  the second sum in \eqref{eq:h2gapGGhat} is completely eliminated; thus leading to the eventual $\mathcal{H}_2\text{-gap}$ error 
$
    \begin{array}{c}
   {\displaystyle 
    \|\bG-\whG \|_{\mathcal{H}_2\text{-gap}}^2= \sum_{i=1}^n \bc _i^T 
\whM(-\lambda_i) \left( \bG(-\lambda_i) - \whG(-\lambda_i) \right)\bb_i }. \end{array}
  $

\begin{algorithm}[ht]
  \caption{gap-IRKA} 
  \label{alg:gap_irka}
  \begin{algorithmic}[1]
    \REQUIRE $\{\sigma_1,\dots,\sigma_r \},$ $\{\Rdir_1,\dots,\Rdir_r\}$ and 
$\{\Ldir_1,\dots,\Ldir_r\}.$    
    \ENSURE $\whA$, $\whB$, $\whC$   \\
    \WHILE{relative change in $\{\sigma_j\}> \mathrm{tol}$} \vspace{1ex}
    \STATE {Compute $\bV_r$ and $\bW_r$ from
    \begin{align*}       
\bV_r&=[(\sigma_1\bI-\bA)^{-1}\bB\Rdir_1,\dots,(\sigma_r\bI-\bA)^{-1}\bB \Rdir_r],\\
\bW_r&=[(\sigma_1\bI-\bA^T)^{-1}\bC^T 
\Ldir_1,\dots,(\sigma_r\bI-\bA^T)^{-1}\bC^T \Ldir_r]. 
    \end{align*}
    \vspace*{-3ex}
    }
   \STATE Perform basis change $\bW_r \leftarrow \bW_r (\bV_r^T\bW_r)^{-1}$   so that
   $\bW_r^T\bV_r = \bI_r$.\vspace{1ex}
    \STATE {Update ROM: $\whA = \bW_r^T \bA 
\bV_r,\whB = \bW_r^T\bB,\whC=\bC \bV_r.$ }\vspace{1ex}
    \STATE {Solve $\whA\whP + \whP \whA^T -\whP \whC^T\whC \whP  
+ \whB \whB^T =0. $}\vspace{1ex}
    \STATE {Compute $\displaystyle\whG_\bF(s)=[\bI,0]+\sum_{j=1}^r \frac{\whc_j 
[\whf_j^T,\whb_j^T]}{s-\widehat{\lambda}_j}. $}\vspace{1ex}
    \STATE {$\sigma_j \leftarrow -\lambda_j$, 
    $\Rdir_j \leftarrow \bb_j$, and $\Ldir_j \leftarrow \bc_j$ for $j=1,2,\ldots,r$. 
    } 
    \ENDWHILE\\    
  \end{algorithmic} 
\end{algorithm}

\subsection{Algorithm \ref{alg:gap_irka} and  $\htwogap$ optimality.}  So far we have motivated Algorithm \ref{alg:gap_irka} as a way to eliminate the contribution from the second-sum in the error expression \eqref{eq:h2gapGGhat}. However, the  reduced model  $\whG$ from Algorithm \ref{alg:gap_irka} achieves more. First, recall that $ \|\bG-\whG \|_{\mathcal{H}_2\text{-gap}}= \| 
\bG_\bF-\whG_\bF\|_{\mathcal{H}_2}$. Therefore, if we interpret the problem as the $\htwo$ optimal model reduction of the closed-loop model $\bG_\bF$, the $\htwo$ optimal reduced closed-loop model
(or equivalently $\htwogap$ optimal closed-loop model)
$\displaystyle\whG_\bF(s)=[\bI,0]+\sum_{j=1}^r \frac{\whc_j 
[\whf_j^T,\whb_j^T]}{s-\widehat{\lambda}_j}$ satisfies 
\begin{align} \nonumber
\left(\bG_\bF(-\widehat{\lambda}_j) - \whG_\bF(-\widehat{\lambda}_j)\right)  \begin{bmatrix} \whf_j\\ \whb_j
\end{bmatrix} = \mathbf{0},
~~\whc_j^T\left(\bG_\bF(-\widehat{\lambda}_j) - \whG_\bF(-\widehat{\lambda}_j) \right) = 0,\\ ~~\mbox{and}~~\whc_j^T\left(\bG_\bF'(-\widehat{\lambda}_j) - \whG_\bF'(-\widehat{\lambda}_j)\right) \begin{bmatrix} \whf_j\\ \whb_j 
\end{bmatrix} = 0,~~\mbox{for}~~j=1,2,\ldots,r. \label{eq:h2optcond}
\end{align}
Therefore, upon convergence, Algorithm \ref{alg:gap_irka} enforces the first set of necessary conditions in \eqref{eq:h2optcond}, namely the right-tangential interpolation conditions. 

One can also think about the necessary and sufficient conditions for the restricted setting. Assume that the reduced closed-loop poles $\{\widehat{\lambda}_j\}$ and the reduced closed-loop left residue-directions
$\{ \whc_j\}$ are fixed. Thus, the only variables in $\whG_\bF$ are the right  residue-directions $\{[\whf_j^T~\whb_j^T]\}$. As \cite{BeaG12} showed,  for fixed reduced poles and left residue-directions, $\whG_\bF$ minimizes  $ \| 
\bG_\bF-\whG_\bF\|_{\mathcal{H}_2} = \|\bG-\whG \|_{\mathcal{H}_2\text{-gap}}$ \emph{if and only if} 
$ \left(\bG_\bF(-\widehat{\lambda}_j) - \whG_\bF(-\widehat{\lambda}_j)\right)  \begin{bmatrix} \whf_j\\ \whb_j
\end{bmatrix} = \mathbf{0}$; thus the right tangential interpolation at the mirror images of the reduced poles become necessary and sufficient conditions for optimality. Therefore,  for the converged reduced closed-loop poles and left residue-directions, Algorithm  \ref{alg:gap_irka} gives the global minimizer.
We end this discussion with a warning. The optimality conditions that we argue Algorithm \ref{alg:gap_irka}
satisfies view $\whG_\bF$ as the variable to minimize the error. Corollary \ref{cor:Ghat_to_GFhat} reveals that we can enforce these optimality conditions via choosing $\whG$ appropriately. However, the 
$\htwogap$ optimality conditions with respect to $\whG$  will be different than those in \eqref{eq:h2optcond}. Those conditions, together with an algorithm to satisfy them, will be the topic of future work.

\section{Numerical examples}
 
 In this section, we present two numerical examples resulting from spatial semidiscretizations of partial differential equations. We compare Algorithm \ref{alg:gap_irka} with the method of LQG balanced truncation as well as the standard version of IRKA. Note however that both examples in fact result in unstable dynamical systems such that the application of IRKA needs further explanation. IRKA is  a method  for optimal $\mathcal{H}_2$ model reduction  of asymptotically stable dynamical systems. 
 However from a computational perspective its implementation  does not prevent one from using it on reducing unstable systems with no poles on the imaginary axis. This has been studied extensively in \cite{sinani2016iterative} showing that IRKA applied to unstable systems with a modest number of unstable poles produces accurate approximations. We have chosen this formulation of IRKA as opposed to the modified version in \cite{MagBG10} for unstable systems since the latter requires a full stable-unstable decomposition of the full model.
 
\subsection{An unstable convection diffusion equation.} \label{cdexample}

The first example is a (scalable) finite-difference discretization of the following controlled convection diffusion equation on the unit square
\begin{align*} 
v_t &= \Delta v-20\cdot \mathrm{sin}(x) v_x  + 50 \cdot v +  \chi_{\omega}\cdot u(t) && \text{in } (0,1)^2 \times (0,T), \\
v(x,0,t)&=v(0,y,t)=v(x,1,t)=v(1,y,t)=0 &&  \text{in } (0,1) \times (0,T), \\
v(x,y,0)&=v_0(x,y)=0 && \text{in } (0,1)^2.
%x(\xi,0)&=\frac{1}{\sqrt{ 0.05 \pi}} \mathrm{exp}\left(-\frac{(\xi-0.25)^2}{0.05}\right) && \text{in } (0,1), 
\end{align*}
Here, by $\chi$ we denote the characteristic function on the control domain $\omega=[0.2,0.3]\times [0.2,0.3]$.  We augment the system by an output variable $y(t)$ corresponding to the mean value in an observable domain 
\begin{align*}
 y(t) = \int_{0.5}^{0.7} \int_{0.7}^{0.9} v(x,y,t) \, \mathrm{d}x \, \mathrm{d}y .
\end{align*}
We present the results for a system of dimension $n=400$ corresponding to a uniform $20\times 20$ grid. The discretized system matrix $\bA$ has 12 eigenvalues in the right half plane but the pairs $(\bA,\bB)$ and $(\bA,\bC)$ satisfy the required stabilizability assumptions.

\begin{center}
%\begin{tabular}{|c||c|c|c|c|c|c|}
\begin{tabular}{|c||c|c|c|}
\multicolumn{4}{c} {$\| \bG_\bF-\whG_\bF \|_{\mathcal{H}_2}$}
\\[1ex]
\hline
$r$ &   IRKA & LQG-BT & gap-IRKA \\ \hline
$1$ & - & $ 5.34 \cdot 10^{-1}  $ & $ 5.34 \cdot 10^{-1} $  \\
$2$ & $ 3.09 \cdot 10^{-1} $ & $ 1.03 \cdot 10^{-1} $ & $ \bf{1.03 \cdot 10^{-1}} $  \\
$3$ & $ 1.02 \cdot 10^{-1} $ & $ 1.51 \cdot 10^{-2} $ & $ \bf{1.47 \cdot 10^{-2}} $  \\
$4$ & $ 1.09 \cdot 10^{-2} $ & $ 7.00 \cdot 10^{-3} $ & $ \bf{5.89 \cdot 10^{-3}} $  \\
$5$ & $ 1.38 \cdot 10^{-3} $ & $ 9.16 \cdot 10^{-4} $ & $ \bf{9.04 \cdot 10^{-4}} $  \\
$6$ & $ 2.23 \cdot 10^{-4} $ & $ 6.91 \cdot 10^{-5} $ & $ \bf{6.83 \cdot 10^{-5}} $  \\
$7$ & $ 6.01 \cdot 10^{-5} $ & $ 5.88 \cdot 10^{-5} $ & $ \bf{5.78 \cdot 10^{-5}} $  \\
$8$ & $ 6.24 \cdot 10^{-6} $ & $ 5.23 \cdot 10^{-6} $ & $ \bf{5.18 \cdot 10^{-6}} $  \\
$9$ & $ 7.78 \cdot 10^{-7} $ & $ 4.08 \cdot 10^{-7} $ & $ \bf{3.98 \cdot 10^{-7}} $  \\
$10$ & $ 3.11 \cdot 10^{-7} $ & $ 3.06 \cdot 10^{-7} $ & $ \bf{3.02 \cdot 10^{-7}} $  \\
$11$ & $ 2.11 \cdot 10^{-8} $ & $ 1.86 \cdot 10^{-8} $ & $ \bf{1.85 \cdot 10^{-8}} $ \\
$12$ & $ 1.45 \cdot 10^{-8}\ $ & $ 2.77 \cdot 10^{-9} $ & $ \bf{2.69 \cdot 10^{-9}} $  \\ \hline
\end{tabular} \ 
%\captionof{table}{$\| \bG_\bF-\whG_\bF \|_{\mathcal{H}_2}$} 
%\begin{tabular}{|c||c|c|c|c|c|c|}
\begin{tabular}{|c||c|c|}
\multicolumn{3}{c} {$\| \bG_\bF-\whG_\bF \|_{\mathcal{H}_\infty}$}
\\[1ex]
\hline
   IRKA & LQG-BT & gap-IRKA \\ \hline
 - & $ 2.29 \cdot 10^{-1}  $ & $ \bf{2.28 \cdot 10^{-1}} $  \\
 $ 5.66 \cdot 10^{-2} $ & $ 1.87 \cdot 10^{-2} $ & $ \bf{1.86 \cdot 10^{-2}} $  \\
 $ 4.80 \cdot 10^{-2} $ & $ 4.75 \cdot 10^{-3} $ & $ \bf{4.28 \cdot 10^{-3}} $  \\
 $ 1.12 \cdot 10^{-3} $ & $ \bf{7.87 \cdot 10^{-4}} $ & $ 2.38 \cdot 10^{-3} $  \\
 $ 1.57 \cdot 10^{-4} $ & $ 9.21 \cdot 10^{-5} $ & $ \bf{7.54 \cdot 10^{-5}} $  \\
 $ 3.70 \cdot 10^{-5} $ & $ \bf{7.59 \cdot 10^{-6}} $ & $ 8.35 \cdot 10^{-6} $  \\
 $ 3.56 \cdot 10^{-6} $ & $ 3.58 \cdot 10^{-6} $ & $ \bf{3.45 \cdot 10^{-6}} $  \\
 $ 4.31 \cdot 10^{-7} $ & $ 3.48 \cdot 10^{-7} $ & $ \bf{3.34 \cdot 10^{-7}} $  \\
 $ 1.01 \cdot 10^{-7} $ & $ \bf{3.99 \cdot 10^{-8}} $ & $\ 4.49 \cdot 10^{-8} $  \\
 $ 1.23 \cdot 10^{-8} $ & $ 1.34 \cdot 10^{-8} $ & $ \bf{1.21 \cdot 10^{-8}} $  \\
 $ 2.39 \cdot 10^{-9} $ & $ 7.19 \cdot 10^{-10} $ & $ \bf{6.41 \cdot 10^{-10}} $ \\
 $ 2.79 \cdot 10^{-9} $ & $ 5.37 \cdot 10^{-10} $ & $ \bf{5.07 \cdot 10^{-10}} $  \\ \hline
\end{tabular}
\captionof{table}{Approximation error $\| \bG_\bF-\whG_\bF \|_{\mathcal{H}_2}$ (left) and $\| \bG_\bF-\whG_\bF \|_{\mathcal{H}_\infty}$ (right).} 
\label{table:conv_diff}
\end{center}

Table \ref{table:conv_diff} shows the error of different reduced-order systems with respect to the gap topology as well as the newly introduced $\mathcal{H}_2$-gap. Note that  in all cases, the reduced-order systems computed via the proposed method gap-IRKA yield smaller $\mathcal{H}_2$-gap errors than LQG balanced truncation as well as IRKA. The results are missing  for $r=1$ for IRKA since it did not converge. Even for the $\mathcal{H}_\infty$-gap, in all but three cases,  gap-IRKA outperforms the other methods. Outperforming LQB balanced truncation with respect to the $\mathcal{H}_\infty$-gap  without ever computing large-scale Riccati solutions is a remarkable demonstration of the effectiveness of the proposed interpolatory framework in reducing unstable systems.

%\end{center}
 
%\caption{Cost of the controls $u_p$.}

%\end{figure}

 \subsection{Linearized Navier-Stokes equations}

In this example, we consider a linearization of the incompressible Navier-Stokes equations around an unsteady flow profile. In particular, we consider the Stokes-Oseen system 
\begin{equation}\label{eq:flow_prob}
  \begin{aligned}
   v_t &= \nu \Delta v  - (v \cdot \nabla)z - (z \cdot \nabla)v- \nabla p + B u &&  \text{in } \Omega \times (0,T), \\
   \mathrm{div}\, v &= 0 &&  \text{in } \Omega \times (0,T), \\
    v&= 0 &&  \text{on } \Gamma \times (0,T), \\
   v(0) &= 0, && \text{in } \Omega ,&&
  \end{aligned}
\end{equation}
where $\nu = \frac{1}{\mathrm{Re}}=\frac{1}{90}$ and the geometry $\Omega=(0,2.2)\times (0.41)$ describes the flow around a cylindric obstacle. The precise setup together with a description of the control and observation operators $B and $, respectively, is given in \cite{BehBH17} which we also refer the reader to for more details. 
We used the semi-discretized model from \cite{BehBH17} corresponding to a Taylor-Hood finite element discretization of the Navier-Stokes equations with $n=n_v+n_p=9356+1289$ degrees of freedom. Since the original systems results in a differential algebraic system, here we explicitly eliminated the pressure term by means of the algebraic approach described in \cite{HeiSS08}. Note that the system matrices of the transformed ODE are dense and an explicit computation generally should be avoided. In our case, the dimension of the ODE is $n=n_v-n_p =8067$ and can still be handled by direct solvers in MATLAB; we refrain from a more sophisticated approach here. 

We repeat the similar experiments as in Example \ref{cdexample} and compare the proposed method to IRKA and 
LQG balanced truncation. The results are depicted in Table \ref{table:flow_problem} where the missing data  for $r=2$ for IRKA is due to non-convergence. As Table \ref{table:flow_problem}  illustrates, the reduced systems generated by Algorithm \ref{alg:gap_irka}   in all cases, except  for one, yield the smallest $\mathcal{H}_2$-gap error. Moreover, eight out of twenty cases tested, the proposed method,  without computing a large-scale Riccati-based Gramians, once again outperforms the LQG balanced truncation in terms of the $\mathcal{H}_\infty$-gap as well despite not being developed for this measure. 

\begin{center}
\begin{tabular}{|c||c|c|c|}
\multicolumn{4}{c} {$\| \bG_\bF-\whG_\bF \|_{\mathcal{H}_2}$}
\\[1ex]
\hline
$r$ &   IRKA & LQG-BT & gap-IRKA \\ \hline
$2$ & - & $ 5.69 \cdot 10^{-1}  $ & $ \bf{5.67\cdot 10^{-1}} $  \\
$4$ & $ 2.42 \cdot 10^{-1} $ & $ 1.99 \cdot 10^{-1} $ & $ \bf{1.93 \cdot 10^{-1}} $  \\
$6$ & $ 1.08 \cdot 10^{-1} $ & $ 1.84 \cdot 10^{-1} $ & $ \bf{1.01 \cdot 10^{-1}} $  \\
$8$ & $ 8.56 \cdot 10^{-2} $ & $ 9.33 \cdot 10^{-2} $ & $ \bf{8.28 \cdot 10^{-2}} $  \\
$10$ & $ 5.78 \cdot 10^{-2} $ & $ 7.02 \cdot 10^{-2} $ & $ \bf{5.63 \cdot 10^{-2}} $  \\
$12$ & $ 3.25 \cdot 10^{-2} $ & $ 3.38 \cdot 10^{-2} $ & $ \bf{2.98 \cdot 10^{-2}} $  \\
$14$ & $ 1.79 \cdot 10^{-2} $ & $ 1.96 \cdot 10^{-2} $ & $ \bf{1.71 \cdot 10^{-2}} $  \\
$16$ & $ \bf{1.12 \cdot 10^{-2}} $ & $ 1.34 \cdot 10^{-2} $ & $ 1.13 \cdot 10^{-2} $  \\
$18$ & $ 7.84 \cdot 10^{-3} $ & $ 9.62 \cdot 10^{-3} $ &  $ \bf{7.19\cdot 10^{-2}} $   \\
$20$ & $ 4.79 \cdot 10^{-3} $ & $ 5.25 \cdot 10^{-3} $ & $ \bf{4.66 \cdot 10^{-3}} $ \\
$22$ & $ 3.52 \cdot 10^{-3} $ & $ 4.29 \cdot 10^{-3} $ & $ \bf{3.44 \cdot 10^{-3}} $ \\
$24$ & $ 2.49 \cdot 10^{-3} $ & $ 2.69 \cdot 10^{-3} $ & $ \bf{2.38 \cdot 10^{-3}} $ \\
$26$ & $ 1.89 \cdot 10^{-3} $ & $ 2.51 \cdot 10^{-3} $ & $ \bf{1.87 \cdot 10^{-3}} $ \\
$28$ & $ 1.39 \cdot 10^{-3} $ & $ 1.89 \cdot 10^{-3} $ & $ \bf{1.38 \cdot 10^{-3}} $ \\
$30$ & $ 1.22 \cdot 10^{-3} $ & $ 1.41 \cdot 10^{-3} $ & $ \bf{1.21 \cdot 10^{-3}} $ \\
$32$ & $ 8.94 \cdot 10^{-4} $ & $ 1.01 \cdot 10^{-3} $ & $ \bf{8.78 \cdot 10^{-4}} $ \\
$34$ & $ 7.01 \cdot 10^{-4} $ & $ 8.52 \cdot 10^{-4} $ & $ \bf{6.89 \cdot 10^{-4}} $ \\
$36$ & $ 5.31 \cdot 10^{-4} $ & $ 6.08 \cdot 10^{-4} $ & $ \bf{5.17 \cdot 10^{-4}} $ \\
$38$ & $ 3.60 \cdot 10^{-4} $ & $ 4.02 \cdot 10^{-4} $ & $ \bf{3.53 \cdot 10^{-4}} $ \\
$40$ & $ 2.51 \cdot 10^{-4} $ & $ 2.85 \cdot 10^{-4} $ & $ \bf{2.45 \cdot 10^{-4}} $ \\   \hline
\end{tabular} \ 
\begin{tabular}{|c||c|c|c|c|c|c|}
\multicolumn{3}{c} {$\| \bG_\bF-\whG_\bF \|_{\mathcal{H}_\infty}$}
\\[1ex]
\hline
   IRKA & LQG-BT & gap-IRKA \\ \hline
 - & $ 3.76 \cdot 10^{-1}  $ & $ \bf{3.70\cdot 10^{-1}} $  \\
  $ 9.84 \cdot 10^{-2} $ & $ 8.26 \cdot 10^{-2} $ & $ \bf{7.52 \cdot 10^{-2}} $  \\
  $ 5.11 \cdot 10^{-2} $ & $ 9.02 \cdot 10^{-2} $ & $ \bf{5.11 \cdot 10^{-2}} $  \\
 $ 5.10 \cdot 10^{-2} $ & $ \bf{4.01 \cdot 10^{-2}} $ & $ 4.86 \cdot 10^{-2} $  \\
 $ 3.01 \cdot 10^{-2} $ & $ \bf{2.47 \cdot 10^{-2}} $ & $ 3.01 \cdot 10^{-2} $  \\
 $ 9.68 \cdot 10^{-3} $ & $ 1.15 \cdot 10^{-2} $ & $ \bf{9.63 \cdot 10^{-3}} $  \\
$7.16 \cdot 10^{-3} $ & $ 7.45 \cdot 10^{-3} $ & $ \bf{6.62 \cdot 10^{-3}} $  \\
 $ 5.15 \cdot 10^{-3} $ & $ \bf{3.73 \cdot 10^{-3}} $ & $ 4.70 \cdot 10^{-3} $  \\
 $ 4.78 \cdot 10^{-3} $ & $ \bf{3.14 \cdot 10^{-3}} $ & $ 4.44 \cdot 10^{-3} $   \\
 $ 2.44 \cdot 10^{-3} $ & $ \bf{1.28 \cdot 10^{-3}} $ & $ 2.38 \cdot 10^{-3} $ \\
 $ 1.12 \cdot 10^{-3} $ & $ \bf{9.55 \cdot 10^{-4}} $ & $ 1.06 \cdot 10^{-3} $ \\
 $ 5.81 \cdot 10^{-4} $ & $ 7.38 \cdot 10^{-4} $ &$ \bf{5.67 \cdot 10^{-4}} $ \\
 $ \bf{5.52 \cdot 10^{-4}} $ & $ 8.33 \cdot 10^{-4} $ & $ 5.58 \cdot 10^{-4} $ \\
 $ 4.41 \cdot 10^{-4} $ & $ \bf{4.39 \cdot 10^{-4}} $ & $  4.44 \cdot 10^{-4} $ \\
 $ 4.39 \cdot 10^{-4} $ & $ \bf{2.91 \cdot 10^{-4}} $ & $ 4.32 \cdot 10^{-4} $ \\
 $ 2.63 \cdot 10^{-4} $ & $ \bf{2.02 \cdot 10^{-4}} $ & $\ 2.60 \cdot 10^{-4} $ \\
 $ 2.47 \cdot 10^{-4} $ & $ \bf{1.62 \cdot 10^{-4}} $ & $ 2.47 \cdot 10^{-4} $ \\
 $ 1.22 \cdot 10^{-4} $ & $ 1.35 \cdot 10^{-4} $ & $ \bf{1.21 \cdot 10^{-4}} $ \\
 $ 7.54 \cdot 10^{-5} $ & $ \bf{6.71 \cdot 10^{-5}} $ & $ 7.53 \cdot 10^{-5} $ \\
 $ 5.30 \cdot 10^{-5} $ & $ 6.20 \cdot 10^{-5} $ & $ \bf{5.12 \cdot 10^{-5}} $ \\   \hline
\end{tabular}
\captionof{table}{Approximation error $\| \bG_\bF-\whG_\bF \|_{\mathcal{H}_2}$ (left) and $\| \bG_\bF-\whG_\bF \|_{\mathcal{H}_\infty}$ (right).} 
\label{table:flow_problem}
\end{center}
\section{Conclusion}

We have presented a new approach for model reduction of linear stabilizable and detectable control systems. Based on the theory of left-coprime factorizations and a newly introduced $\mathcal{H}_2$-gap, we have derived pole-residue formulae that suggest tangentially interpolating the original transfer function at the mirrored closed-loop reduced system poles. Since these are not known a priori, we modified the iterative rational Krylov algorithm accordingly. Two numerical examples associated with (unstable) partial differential equations illustrate the applicability and good performance of the new approach.
% Future work will focus on deriving first-order necessary optimality conditions as well as suitable numerical methods.

 %%%%%%%%%%%%%%%%%%%%%%%%%%%%%
\section*{Acknowledgment}
%The work of Breiten was supported in parts by \textcolor{red}{add grant info}.
Parts of this work were completed while the first author visited Virginia Tech; the kind hospitality was greatly appreciated.
The work of Beattie was supported in parts by NSF through Grant DMS-1819110.
The work of Gugercin was supported in parts by NSF through Grants DMS-1720257 and DMS-1819110. 
 \bibliographystyle{spmpsci}
\bibliography{references}

\begin{thebibliography}{10}
\providecommand{\url}[1]{{#1}}
\providecommand{\urlprefix}{URL }
\expandafter\ifx\csname urlstyle\endcsname\relax
  \providecommand{\doi}[1]{DOI~\discretionary{}{}{}#1}\else
  \providecommand{\doi}{DOI~\discretionary{}{}{}\begingroup
  \urlstyle{rm}\Url}\fi

\bibitem{AniBGA13}
Ani{\'c}, B., Beattie, C., Gugercin, S., Antoulas, A.: Interpolatory
  weighted-$\mathcal{H}_2$ model reduction.
\newblock Automatica \textbf{49}(5), 1275--1280 (2013)

\bibitem{AntBG20}
Antoulas, A., , Beattie, C., Gugercin, S.: Interpolatory Methods for Model
  Reduction.
\newblock Society for Industrial and Applied Mathematics, to appear (2020)

\bibitem{Ant05}
Antoulas, A.: Approximation of Large-Scale Dynamical Systems.
\newblock Society for Industrial and Applied Mathematics (2005).
\newblock \doi{10.1137/1.9780898718713}

\bibitem{AntBG10}
Antoulas, A.C., Beattie, C.A., Gugercin, S.: Interpolatory model reduction of
  large-scale dynamical systems.
\newblock In: Efficient Modeling and Control of Large-Scale Systems, pp. 3--58.
  Springer (2010)

\bibitem{barrachina2005parallel}
Barrachina, S., Benner, P., Quintana-Ort{\'\i}, E.S., Quintana-Ort{\'\i}, G.:
  Parallel algorithms for balanced truncation of large-scale unstable systems.
\newblock In: Proceedings of the 44th IEEE Conference on Decision and Control,
  pp. 2248--2253. IEEE (2005)

\bibitem{beattie2007kbm}
Beattie, C., Gugercin, S.: {Krylov}-based minimization for optimal
  {$\mathcal{H}_2$} model reduction.
\newblock In: Proceedings of 46th IEEE Conference on Decision and Control, pp.
  4385--4390 (2007)

\bibitem{BeaG12}
Beattie, C., Gugercin, S.: Realization-independent
  $\mathcal{H}_2$-approximation.
\newblock In: Proceedings of 51st IEEE Conference on Decision and Control, pp.
  4953 -- 4958 (2012)

\bibitem{BehBH17}
Behr, M., Benner, P., Heiland, J.: Example setups of {N}avier-{S}tokes
  equations with control and observation: Spatial discretization and
  representation via linear-quadratic matrix coefficients.
\newblock Tech. rep., Max Planck Institute for Complex Dynamical Systems
  (2017).
\newblock Available from \url{https://arxiv.org/abs/1707.08711}

\bibitem{benner2016}
Benner, P., Saak, J., Uddin, M.M.: Balancing based model reduction for
  structured index-2 unstable descriptor systems with application to flow
  control.
\newblock Numerical Algebra, Control \& Optimization \textbf{6}, 1--20 (2016).
\newblock \doi{10.3934/naco.2016.6.1}.
\newblock
  \urlprefix\url{http://aimsciences.org//article/id/01f6d0e6-4c13-44fd-9471-8e3bf8978554}

\bibitem{BreBG15}
Breiten, T., Beattie, C., Gugercin, S.: Near-optimal frequency-weighted
  interpolatory model reduction.
\newblock Systems \& Control Letters \textbf{78}, 8--18 (2015)

\bibitem{Bunetal10}
Bunse-Gerstner, A., Kubalinska, D., Vossen, G., Wilczek, D.: {h2}-norm optimal
  model reduction for large scale discrete dynamical {MIMO} systems.
\newblock Journal of Computational and Applied Mathematics \textbf{233}(5),
  1202--1216 (2010).
\newblock \doi{10.1016/j.cam.2008.12.029}

\bibitem{Cur03}
Curtain, R.: Model reduction for control design for distributed parameter
  systems.
\newblock In: R.~Smith, M.~Demetriou (eds.) Research Directions in Distributed
  Parameter Systems, pp. 95--121. SIAM (2003)

\bibitem{GalVVD05}
Gallivan, K., Vandendorpe, A., Van~Dooren, P.: Model reduction of {MIMO}
  systems via tangential interpolation.
\newblock SIAM Journal on Matrix Analysis and Applications \textbf{26}(2),
  328--349 (2005)

\bibitem{GawJ90}
Gawronski, W., Juan, J.N.: Model reduction in limited time and frequency
  intervals.
\newblock International Journal of Systems Science \textbf{21}(2), 349--376
  (1990).
\newblock \doi{10.1080/00207729008910366}

\bibitem{Gug03}
Gugercin, S.: Projection methods for model reduction of large-scale dynamical
  systems.
\newblock Ph.D. thesis, Rice University (2003)

\bibitem{GugAB08}
Gugercin, S., Antoulas, A., Beattie, C.: $\mathcal{H}_2$ model reduction for
  large-scale linear dynamical systems.
\newblock SIAM Journal on Matrix Analysis and Applications \textbf{30}(2),
  609--638 (2008)

\bibitem{GugA03}
{Gugercin}, S., {Antoulas}, A.C.: A time-limited balanced reduction method.
\newblock In: 42nd IEEE International Conference on Decision and Control,
  vol.~5, pp. 5250--5253 (2003)

\bibitem{BeaG17}
Gugercin, S., Beattie, C.: Model reduction by rational interpolation.
\newblock In: Model Reduction and Approximation, pp. 297--334. SIAM,
  Philadelphia (2017)

\bibitem{HeiSS08}
Heinkenschloss, M., Sorensen, D., Sun, K.: Balanced truncation model reduction
  for a class of descriptor systems with application to the {O}seen equations.
\newblock SIAM Journal on Scientific Computing \textbf{30}(2), 1038--1063
  (2008)

\bibitem{JonS83}
{Jonckheere}, E., {Silverman}, L.: A new set of invariants for linear systems
  -- application to reduced order compensator design.
\newblock IEEE Transactions on Automatic Control \textbf{28}(10), 953--964
  (1983)

\bibitem{Kue18}
K{\"u}rschner, P.: Balanced truncation model order reduction in limited time
  intervals for large systems.
\newblock Advances in Computational Mathematics \textbf{44}(6), 1821--1844
  (2018).
\newblock \doi{10.1007/s10444-018-9608-6}

\bibitem{LinS15}
Lin, Y., Simoncini, V.: A new subspace iteration method for the algebraic
  {R}iccati equation.
\newblock Numerical Linear Algebra with Applications \textbf{22}(1), 26--47
  (2015).
\newblock \doi{10.1002/nla.1936}

\bibitem{MagBG10}
{Magruder}, C., {Beattie}, C., {Gugercin}, S.: Rational {K}rylov methods for
  optimal $\mathcal{L}_2$ model reduction.
\newblock In: 49th IEEE Conference on Decision and Control (CDC), pp.
  6797--6802 (2010)

\bibitem{Mey90}
Meyer, D.: Fractional balanced reduction: model reduction via fractional
  representation.
\newblock IEEE Transactions on Automatic Control \textbf{35}(12), 1341--1345
  (1990)

\bibitem{NetJB84}
{Nett}, C., {Jacobson}, C., {Balas}, M.: A connection between state-space and
  doubly coprime fractional representations.
\newblock IEEE Transactions on Automatic Control \textbf{29}(9), 831--832
  (1984)

\bibitem{Par91}
Partington, J.: Approximation of unstable infinite-dimensional systems using
  coprime factors.
\newblock Systems \& Control Letters \textbf{16}, 89--96 (1991)

\bibitem{sinani2016iterative}
Sinani, K.: Iterative rational krylov algorithm for unstable dynamical systems
  and genaralized coprime factorizations.
\newblock Master's thesis, Virginia Tech (2016)

\bibitem{SinG18}
Sinani, K., Gugercin, S.: $\mathcal{H}_2(t_f)$ optimality conditions for a
  finite-time horizon.
\newblock Automatica, accepted to appear  (2019).
\newblock Available from \url{https://arxiv.org/abs/1806.10797}

\bibitem{Vid84}
Vidyasagar, M.: The graph metric for unstable plants and robustness estimates
  for feedback stability.
\newblock IEEE Transactions on Automatic Control \textbf{29}(5), 403--418
  (1984)

\bibitem{xu2010optimal}
Xu, Y., Zeng, T.: Optimal $\htwo$ model reduction for large scale {MIMO}
  systems via tangential interpolation.
\newblock International Journal of Numerical Analysis and Modeling
  \textbf{8}(1), 174--188 (2011)

\bibitem{zhou1999balanced}
Zhou, K., Salomon, G., Wu, E.: Balanced realization and model reduction for
  unstable systems.
\newblock International Journal of Robust and Nonlinear Control \textbf{9}(3),
  183--198 (1999)

\end{thebibliography}

%\newpage
 %\section*{\textcolor{blue}{The main messages}}
%\begin{enumerate}
%\item Introduce the $H_2$-gap norm
%\item Proposition \ref{prop:l2_bound}: $H_2$-gap implies bound for $\| G - \hat{G}\|_{L_2}$
%\item  Proposition \ref{prop:cph2_res}: Extends pole-residue $H_2$ formula to $H_2$-gap
%\item All these expressions use closed-loop quantities $G_F$ and $\hat{G}_F$ 
%\item Lemmas \ref{lem:res_and_G} -- \ref{lem:G_vs_GP2} established important quantities to set the stage for Theorem 
%\ref{thm:main_result}
%\item  Theorem  \ref{thm:main_result} eliminates $G_F$ and $\hat{G}_F$  from the error formula
%\item This allows, via Corollary \ref{cor:Ghat_to_GFhat}, interpolating the closed-loop system without computing $P$
%\item Then it motivates the algorithm
%\end{enumerate}

\end{document}